\documentclass[12pt]{amsart}
\usepackage[T1]{fontenc}
\usepackage{ dsfont}
\usepackage{amsmath,amsthm, amssymb, latexsym,}
\usepackage{cases}
\usepackage{mathrsfs}
\usepackage{hyperref}
\usepackage{pifont}
\usepackage{amsfonts}
\usepackage{stmaryrd}
\usepackage{color}
\usepackage[all]{xy}
\usepackage[english]{babel}
\usepackage[latin1]{inputenc}

\usepackage[all]{xy}
\newtheorem{thmIntro}{Theorem}    
\newtheorem{theorem}{Theorem}[section]
\newtheorem{lemma}[theorem]{Lemma}

\newtheorem{proposition}[theorem]{Proposition}

\theoremstyle{definition}
\newtheorem{definition}[theorem]{Definition}

\theoremstyle{remark}
\newtheorem{remark}[theorem]{Remark}

\numberwithin{equation}{section}

\newcommand{\on}[1]{\operatorname{#1}}

\newcommand{\xg}{\backslash} % xie gang
\newcommand{\lieg}{\mathfrak{g}} % lie algebra g
\newcommand{\liel}{\mathfrak{l}} % levi subalgebra l

\newcommand{\ten}{\otimes} % tensor
\newcommand{\fz}{\mathbb{Z}} % field symbol of integer number Z
\newcommand{\fn}{\mathbb{N}} % field symbol of natural number N
\newcommand{\fk}{\mathbb{K}} % field symbol of natural number N
\newcommand{\ii}{\mathbb{I}} % index symbol of arbitrary I
\newcommand{\ra}{\rightarrow} % "right arrow" in a map
\newcommand{\mt}{\mapsto} % map with elements, mt means "map to"
\begin{document}
\title[Centers of Universal Enveloping Algebras]{Centers of Universal Enveloping Algebras}
\author{Yaping Yang} \address[Yaping Yang]{School of Mathematics, Yunnan Normal University, Kunming 650500, China} \email{yypiwan@126.com}

\author{Daihao Zeng}
\address[Daihao Zeng]{School of Statistics and Mathematics, Yunnan University of Finance and Economics, Kunming 650221, China} \email{202202110812@stu.ynufe.edu.cn}
%\thanks{Correspondence Author: Ke Ou}
\date{\today}
\begin{abstract}
The universal enveloping algebra $U(\lieg)$ of a current (super)algebra or loop (super)algebra $\lieg$ is considered over an algebraically closed field $\fk$ with characteristic $p\geq 0$. This paper focuses on the structure of the center $Z(\lieg)$ of $U(\lieg)$. In the case of zero characteristic, $Z(\lieg)$ is generated by the centers of $\lieg$. In the case of prime characteristic, $Z(\lieg)$ is generated by the centers of $\lieg$ and the $p$-centers of $U(\lieg)$. We also study the structure of $Z(\lieg)$ in the semisimple Lie (super)algebra.
\end{abstract}
\maketitle
\section{Introduction}

Let $\liel$ be a Lie (super)algebra over a field $\fk$. Its universal enveloping algebra $U(\liel)$ is an associative (super)algebra with unity together with a homomorphism of Lie (super)algebras $\iota:\liel \mapsto U(\liel)$ characterized by the universal property. It associates the representations of a Lie (super)algebra to representations of an associative (super)algebra (see \cite{CW,Hu2}). The center $Z(\liel)$ of $U(\liel)$ is defined as the set $$\{z \in U(\liel) \mid [z,u]=0, \forall u\in U(\liel)\},$$
where $[z,u]=zu-uz$ (resp. $[z,u]=zu-(-1)^{\mid z\mid \cdot \mid u\mid}uz$) if $\liel$ is a Lie algebra (resp. Lie superalgebra). Let $\lieg$ be the current algebra of the finite-dimensional classical simple Lie algebras $\liel$ over a complex field, and $\lieg$ has a Lie algebraic structure (see Definition \ref{2.0}). Molev discussed the connection between the center $Z(\lieg)$ of $U(\lieg)$ and the center of the Yangians (see \cite[1.7]{Mo}). Subsequently, more scholars have investigated this connection and used it as a foundation for calculating the centers of specific (super) Yangians (see \cite{BT,CH}).

In zero characteristic, it is known that the center of the universal enveloping algebra of a reductive Lie algebra in a complex field is isomorphic to its homomorphic image under twisted Harish-Chandra homomorphisms (see \cite[1.10]{Hu1}). Konno (\cite{Ko}) provided a precise description of the Harish-Chandra homomorphic image of $\mathfrak{sl}(m,1)$ in the case of Lie superalgebras. For a finite-dimensional nilpotent Lie algebra $\liel$ over a field with zero characteristic, Dixmier (\cite{Di}) proves that $Z(\liel)$ is a unique factorization domain and Moeglin (\cite{Moe}) proves that each height one prime ideal in $U(\liel)$ is generated by a central element. For finite-dimensional classical simple Lie algebras $\liel$ over an algebraically closed field of characteristic zero, Gauger (\cite{Ga}) obtains algebraically independent generators of $Z(\liel)$.

In prime characteristic, Veldkamp (\cite{Ve}) proves that $Z(\liel)$ is generated by the generators of the $p$-center and the Harish-Chandra center, which $\liel$ is the Lie algebra of a split semisimple algebraic group $G$ over a perfect field $\fk$ with characteristic $p>0$ satisfying a condition. Many subsequent studies of $Z(\liel)$ have Veldkamp's theorem in mind. Braun (\cite{Br}) generalized the works of Dixmier (\cite{Di}) and Moeglin (\cite{Moe}) to the prime characteristic field. For a finite-dimensional solvable Lie algebra $\liel$ over a prime characteristic field, Braun and Vernik (\cite{BV}) obtain several equivalent conclusions for $Z(\liel)$ being the unique factorization domain. Further, when $\liel$ is a Lie algebra of a connected reductive group over an algebraically closed field of prime characteristic, R. Tange (\cite{Ta}) proves that $Z(\liel)$ is the unique factorization domain. Brundan and Topley (\cite{BT}) obtained the free generators of the centers of the universal enveloping algebra of shifted current algebra of type A over a prime characteristic field. H. Chang and H. Hu (\cite{CH}) discusses the free generators of the center $Z(\mathfrak{gl}_{m\mid n}[x])$ of $U(\mathfrak{gl}_{m\mid n}[x])$.

Let $\lieg$ denote the current (super)algebra or loop (super)algebra of Lie (super)algebra $\liel$. The main goal of this paper is to obtain the generators of the center $Z(\lieg)$. Our main results are as follows.

\begin{thmIntro}\label{1.1}
If $\liel$ is a Lie (super)algebra over $\fk$ with $\fk=\overline{\fk}$ and $\on{ch}(\fk)=0$, then the center $Z(\lieg)$ of $U(\lieg)$ is equal to $\fk[ e_{jr}\mid j\in J, r\in\ii ]$. In particular, if $C(\liel)\neq \left \{ 0 \right \}$, $Z(\lieg)$ is freely generated by $\{ e_{jr}\mid j\in J, r\in\ii \}$. 
\end{thmIntro}

\begin{thmIntro}\label{1.2}
If $\liel$ is a Lie (super)algebra over $\fk$ with $\fk=\overline{\fk}$ and $\on{ch}(\fk)=p>0$, then the center $Z(\lieg)  $ is generated by $\{ e_{jr}\mid j\in J, r\in \ii \}$ and $Z_p(\lieg). $  
$\\ $
(1) If $\liel$ is a Lie algebra, then $Z(\lieg)$ is freely generated by 
\begin{equation}
	\{ e_{jr}\mid j\in J, r\in \ii \}\cup \{ e_{ir}^p-e_{ir}^{[p]}\mid i\in I\xg J, r\in \ii \};
 \nonumber 
 \end{equation}
 (2) If $\liel$ is a Lie superalgebra, then $Z(\lieg)$ is freely generated by 
 \begin{equation}
	\{ e_{jr}\mid j\in J, r\in \ii \}\cup \{ e_{ir}^p-e_{ir}^{[p]}\mid i\in I_0\xg J, r\in \ii \}.
 \nonumber 
 \end{equation}
\end{thmIntro}

See subsection \ref{2CL} for a description of the relevant symbols in Theorem \ref{1.1} and Theorem \ref{1.2}.

The proofs of both Theorem \ref{1.1} and Theorem \ref{1.2} depend on (\cite{BT}) and (\cite{CH}). In Theorem \ref{1.2}, item (1) (resp. item (2)) generalizes the results in (\cite{BT}) (resp. (\cite{CH})) (see Remark \ref{4.1.04} and \ref{4.2.10}). The Harish-Chandra center of $U(\lieg)$ is trivial if $\lieg$ is a semisimple Lie (super)algebra (see Proposition \ref{3.1.3}, \ref{3.2.3}, \ref{4.1.3} and \ref{4.2.3}).

This paper is organized as follows.
Section \ref{2} introduces some necessary background and specifies the notation used. Sections \ref{3} and \ref{4} provide detailed discussions on the generators of $Z(\lieg)$, where Section \ref{3} focuses on the case of zero characteristic and Section \ref{4} considers the case of prime characteristic. Both Sections \ref{3} and \ref{4} give applications of the main theorem to semisimple Lie (super)algebra.

Throughout this paper, $\fk$ denotes an algebraically closed field with characteristic $p\geq 0$.

\section{Notations and preliminaries}\label{2}
\subsection{The current (super)algebra and loop (super)algebra}\label{2CL}

\begin{definition}
\label{2.0}
(\cite[7.1]{Ka}\footnote{\cite[7.1]{Ka} gives only the definition of the loop algebra, and the definition of the current algebra is similar.})
For a Lie algebra $\liel$ over $\fk$, $\liel[x]:=\liel\ten \fk[x]$ (resp. $\liel(x):=\liel\ten \fk[x,x^{-1}]$) is a Lie algebra with Lie bracket $[e\ten x^r, f\ten x^s]= [e,f]\ten x^{r+s}$ (resp. $[e\ten x^u, f\ten x^v]= [e,f]\ten x^{u+v}$) for each $e,f\in \liel $ and $r,s\in \fn$ (resp. $e,f\in \liel $, $u,v\in \fz$), which is called the {\em current algebra} (resp. {\em loop algebra}) of $\liel$. 
\end{definition}

 \begin{definition}\label{2.01}(\cite[2.1]{CH}\footnote{\cite[2.1]{CH} gives only the definition of the current superalgebra, and the definition of the loop superalgebra is similar.})
For a Lie superalgebra $\liel$ over $\fk$, $\liel[x]:=\liel\ten \fk[x]$ (resp. $\liel(x):=\liel\ten \fk[x,x^{-1}]$) is a Lie superalgebra with the $\fz_2$-grading is defined by $|e\ten x^r|=|e|$ (resp. $|e\ten x^u|=|e|$) and the Lie bracket is given by $[e\ten x^r, f\ten x^s]= [e,f]\ten x^{r+s}$ (resp. $[e\ten x^u, f\ten x^v]= [e,f]\ten x^{u+v}$) for each $e,f\in \liel $ and $r,s\in \fn$ (resp. $e,f\in \liel $, $u,v\in \fz$), which is called the {\em current superalgebra} (resp. {\em loop superalgebra}) of $\liel$. 
\end{definition}
Entail the whole paper without special mention. Let $\lieg=\liel[x]$ or $\liel(x)$, where $\liel$ is a Lie algebra or Lie superalgebra, and denote $ \ii=\fn $ if $\lieg= \liel[x]$ or $ \ii=\fz $ if $\lieg= \liel(x)$. Recall that $C(\liel)$ is the center of $\liel$ while $Z(\liel)$ is the center of $U(\liel)$.

Let $\{e_i\mid i\in I\} $ be a basis of $\liel$ such that $\{e_j\mid j\in J\} $ is a basis of $C(\liel)$ where $J\subseteq I.$ In particular, $C(\liel)$ has trivial center if and only if $J=\varnothing$.
If $\liel$ is a Lie superalgebra, since the center $C(\liel)$ consists of only even elements, we can assume that $\{e_k\mid k\in I_0\}$ is a basis of $\liel_{\overline{0}}$ and $\{e_l\mid l\in I_1\}$ is a basis of $\liel_{\overline{1}}$ such that $I=I_0\cup I_1$ and $J\subseteq I_0$ (see \cite[2.2.2]{CW}). Denote $e_{ir}:=e_i\ten x^r \in \lieg$ with $i\in I, r\in \ii $.

\begin{lemma}\label{2.3}
    For a Lie algebra $\liel$ over $\fk$, we have that $C(\liel)\ten \fk[x]\subseteq Z(\lieg)$ if $\lieg=\liel[x]$ and $C(\liel)\ten \fk[x,x^{-1}]\subseteq Z(\lieg)$ if $\lieg=\liel(x).$
\end{lemma}
\begin{proof}
     Suppose $\lieg=\liel[x]$. The Lie bracket $[e\ten x^r, f\ten x^s]= [e,f]\ten x^{r+s}$ for each $e,f\in \liel $ and $r,s\in \fn$ implies that $C(\liel)\ten \fk[x]\subseteq Z(\lieg)$.

     The arguments are similar if $\lieg=\liel(x)$ which will be omitted here.
\end{proof}

\begin{lemma}\label{2.4}
     For a Lie algebra $\liel$ over $\fk$, we have that $\{ e_{ir}\mid i\in I, r\in \ii \}$ is a basis of $\lieg$.
\end{lemma}
\begin{proof}
   It is easy to know that $\lieg$ is $\fk$-spanned by $\{ e_{ir}\mid i\in I, r\in \ii \}$ since $\liel$ is $\fk$-spanned by $\{e_i\mid i\in I\}$. Let \begin{equation}\label{2.2}
       \sum_{(i,r)\in I\times \ii} b_{ir} e_{ir}=0,
   \end{equation}
    where $b_{ir}\in \fk$ and $b_{ir}=0$ for all $(i,r)\in I\times \ii$ except finite many.
    (\ref{2.2}) implies that $$\sum_{i\in I} b_{ir} e_{i}=0,$$ for each $r\in \ii$. As $\{e_i\mid i\in I\} $ is a basis of $\liel$, we conclude that $b_{ir}=0$ for all $(i,r)\in I\times \ii$ which implies that $\{ e_{ir}\mid i\in I, r\in \ii \}$ is a basis of $\lieg$.
\end{proof}

 The proof of the following conclusion is omitted here, which is analogous to Lie algebra.

\begin{lemma}\label{2.5}
    For a Lie superalgebra $\liel$ over $\fk$, we have that $C(\liel)\ten \fk[x]\subseteq Z(\lieg)$ if $\lieg=\liel[x]$ and $C(\liel)\ten \fk[x,x^{-1}]\subseteq Z(\lieg)$ if $\lieg=\liel(x).$
\end{lemma}

\begin{lemma}\label{2.6}
     For a Lie superalgebra $\liel$ over $\fk$, we have $\{ e_{ir}\mid i\in I, r\in \ii \}$ is a basis of $\lieg$. In particular, $\{ e_{kr}\mid i\in I_0, r\in \ii \}$ (resp. $\{ e_{lr}\mid i\in I_1, r\in \ii \}$) is a basis of $\lieg_{\overline{0}}$ (resp. $\lieg_{\overline{1}}$).
\end{lemma}

Suppose $\{a_{ijk}\in \fk\mid i,j,k\in I\}$ are the structure coefficients of $\liel$. Namely, $$[e_i,e_j]=\sum_{k\in I} a_{ijk} e_k,$$ for all $i,j\in I.$
For each $i,j\in I,$ it follows that $a_{ijk}=0$ for all $k\in I$ except finite many. In particular, if $i\in J$ or $j\in J$, then $a_{ijk}=0$ for all $k\in I.$ Moreover, $$ [e_{ir}, e_{js}]=\sum_{k\in I}a_{ijk}e_{k,r+s} ,$$ for all $i,j\in I$ and $r,s\in \ii$. We can know that $\{a_{ijk}\in \fk\mid i,j,k\in I\}$ are also the structure coefficients of $\lieg$ since Lemma \ref{2.4} and Lemma \ref{2.6}. Define $\on{deg}_x(e_{i_1r_1}\cdots e_{i_tr_t})=r_1+\cdots +r_t$ the $x$-degree of the monomial $e_{i_1r_1}\cdots e_{i_tr_t}\in S(\lieg).$

In prime characteristic, the current (super)algebra, as well as loop (super)algebra $\lieg $ admits a natural structure of restricted Lie (super)algebra, with $p$-th power map $x\mt x^{[p]}$ defined by $(e_i\ten x^r)^{[p]}=e_i^{[p]}\ten x^{rp}$ for all $i\in I$ and $r\in \ii$. Consequently, if $\liel$ is a Lie algebra, $U(\lieg)$ has a large $p$-center generated by the elements
$$\left\{ e_{ir}^p-e_{ir}^{[p]}\mid i\in I, r\in \ii \right\}.$$ 
If $\liel$ is a Lie superalgebra, $U(\lieg)$ also has a large $p$-center generated by the elements
$$\left\{ e_{ir}^p-e_{ir}^{[p]}\mid i\in I_0, r\in \ii \right\}.$$ 

\subsection{Symmetric invariants}
Let $S(\lieg)$ be the symmetric algebra (resp. supersymmetric algebra) if $\liel$ is a Lie algebra (resp. Lie superalgebra). There is a filtration $$U(\lieg)=\bigcup_{r\in \ii} \on{F}_rU(\lieg)$$ of the enveloping algebra $U(\lieg)$ which is defined by placing $e_{ir}$ in degree $r+1,$ that is, $\on{F}_rU(\lieg)$ is the span of all monomials of the form $e_{i_1r_1}\cdots e_{i_sr_s}$ with total degree $(r_1+1)+\cdots+(r_s+1)\leq r.$ The associated graded algebra $\on{gr}U(\lieg)$ is isomorphic (both as graded algebra and as a graded $\lieg$-module) to $S(\lieg).$ If $\on{ch}(\fk)>0,$ the degree of $e_{ir}^p$ (resp. $ e_i^{[p]} \ten x^{rp} $) is $(r+1)p$ (resp. $rp+1$), and hence $\on{gr}(e_{ir}^p-e_{ir}^{[p]})=e_{ir}^p.$ It follows that
\begin{equation}\label{2.1}
 \on{gr}Z(\lieg)\subseteq S(\lieg)^{\lieg}.
\end{equation}
\begin{remark}
(\ref{2.1}) is independent of the characteristic of $\fk$.
\end{remark}

\section{In zero characteristic}\label{3}

In this section, the characteristic of $\fk$ is zero.

\subsection{Lie algebra}
\begin{lemma}\label{3.1.1}
	If $\liel$ is a Lie algebra over $\fk$, then the invariant algebra $S(\lieg)^{\lieg}$ is equal to $\fk[ e_{jr}\mid j\in J, r\in\ii ]$. In particular, if $C(\liel)\neq \left \{ 0 \right \}$, $S(\lieg)^{\lieg}$ is freely generated by $\{ e_{jr}\mid j\in J, r\in \ii \}.$
\end{lemma}
\begin{proof}
	The method of proof is highly depends on \cite[Lemma 3.2]{BT}. If $C(\liel)\neq \left \{ 0 \right \}$. It is obvious that $e_{jr}\in S(\lieg)^{\lieg} $ for each $j\in J$ and $r\in \ii$. Let $I(\lieg)=\fk[ e_{jr}\mid j\in J, r\in\ii ]$, then $I(\lieg)$ is a subalgebra of $S(\lieg)^{\lieg}$ generated by $\{e_{jr}\mid j\in J, r\in \ii\}.$ Since the elements $\{e_{ir}\mid i\in I,r\in \ii\}$ give a basis for $\lieg$, it follows that
  \[ S(\lieg)=\fk[ e_{jr}\mid j\in J, r\in\ii ][e_{ir}\mid i\in I\xg J, r\in\ii], \]
	\[ I(\lieg)=\fk[ e_{jr}\mid j\in J, r\in\ii ] \]
	with both being free polynomial algebras. Hence, $S(\lieg)$ is free as an $I(\lieg)$ module with basis $$\left\{ \prod_{i\in I\xg J, r\in \ii} e_{ir}^{	\omega(i,r)} \mid \omega\in \Omega_1 \right\},$$
	where $\Omega_1=\{\omega: I\xg J\times \ii \ra \fn\mid \omega(i,r)=0 \text{ for all but finite many }(i,r)\in I\times \ii \}.$
	
	Given $f\in S(\lieg)^{\lieg}, $ we can write
	\[ f=\sum_{\omega\in \Omega_1} c_{\omega} \prod_{i\in I\xg J, r\in \ii}e_{ir}^{\omega(i,r)} \]
	for all $c_{\omega}\in I(\lieg)$, all but finitely many of which are zero. In order to prove $S(\lieg)^{\lieg}\subseteq I(\lieg),$ we have to prove that $c_{\omega}=0$ for all non-zero $\omega.$
	
	Now, fix a non-zero $\omega$ appearing in the expression of $f$ such that $\omega(j,r)>0$ for $j\in I\xg J, r\in \ii.$ If $J=I$, we have that $S(\lieg)=I(\lieg)\subseteq S(\lieg)^{\lieg}\subseteq S(\lieg)$, so $S(\lieg)^{\lieg}=I(\lieg)$. Let $J\subsetneqq I$, since $e_j\not\in C(\liel), $ there is $i\in I\xg J$ such that $[e_i,e_j]\neq 0.$ Namely, $a_{ijk}\neq 0$ for some $k\in I.$
	For each integer $s>\on{deg}_x(f),$ which implies that $s$ is larger than all $r^\prime \in\ii$ such that $\omega(i^\prime,r^\prime)>0$ for $(i^\prime,r^\prime)\in I\xg J\times \ii,$ we have
	\begin{multline}\label{3.1}
      \on{ad}(e_{is})(f)= \sum_{\omega\in \Omega_1} c_{\omega} \sum_{\substack{j\in I\xg J, r\in \ii\\ \omega(j,r)>0} }\left (\sum_{k\in I\xg J} a_{ijk}  \omega(j,r)e_{k,r+s} +\sum_{k\in J} a_{ijk}  \omega(j,r)e_{k,r+s} \right ) \\ \times  e_{jr}^{\omega(j,r)-1}\prod_{\substack{j^\prime\in I\xg J, r^\prime\in \ii\\(j^\prime,r^\prime)\neq (j,r)}}e_{j^\prime r^\prime}^{\omega(j^\prime,r^\prime)}.  
	\end{multline}

 Thanks to the choice of $s$, the coefficient of
\[ e_{k,r+s} e_{jr}^{\omega(j,r)-1} \prod_{\substack{j^\prime\in I\xg J, r^\prime\in \ii\\(j^\prime,r^\prime)\neq (j,r)}}e_{j^\prime r^\prime}^{\omega(j^\prime,r^\prime)} \]
is $ a_{ijk} c_{\omega} \omega(j,r) $ for all $k\in I\xg J,$ and the coefficient of
	\[  e_{jr}^{\omega(j,r)-1} \prod_{\substack{j^\prime\in I\xg J, r^\prime\in \ii\\(j^\prime,r^\prime)\neq (j,r)}}e_{j^\prime r^\prime}^{\omega(j^\prime,r^\prime)} \]
	is $ a_{ijk} c_{\omega} \omega(j,r)e_{k,r+s} $ for all $k\in J.$ They must be zero since $f\in S(\lieg)^{\lieg}.$ As $\omega(j,r)>0$ and $ a_{ijk}\neq 0 $ for some $k\in I,$ we conclude that $c_{\omega}=0$ which implies that $S(\lieg)^{\lieg}\subseteq I(\lieg)$.

 The arguments are similar if $C(\liel)=\left \{ 0 \right \}$ which will be omitted here.
\end{proof}

\begin{theorem}\label{3.1.2}
If $\liel$ is a Lie algebra over $\fk$, then the center $Z(\lieg)$ of $U(\lieg)$ is equal to $\fk[ e_{jr}\mid j\in J, r\in\ii ]$. In particular, if $C(\liel)\neq \left \{ 0 \right \}$, $Z(\lieg)$ is freely generated by $\{ e_{jr}\mid j\in J, r\in\ii \}$.
\end{theorem}
\begin{proof}
	The proof, which uses Lemma \ref{3.1.1} and (\ref{2.1}), is similar to the proof of \cite[Theorem 3.4]{BT}. Let $Z=\fk[ e_{jr}\mid j\in J, r\in\ii ]$, then $Z$ is a subalgebra of $Z(\lieg)$. We want to prove that $S(\lieg)^{\lieg}\subseteq\on{gr}Z $. If $C(\liel)=\left \{ 0 \right \}$, it follows from Lemma \ref{3.1.1} that $S(\lieg)^{\lieg}=\fk=\on{gr}Z$.
 
   Let $C(\liel)\neq \left \{ 0 \right \}$, then $Z$ is generated by $\{ e_{jr}\mid j\in J, r\in\ii \}$. For $j\in J$ and $r\in \ii$ we have that $e_{jr}\in \on{F}_{r+1}U(\lieg),$ so $$\on{gr}(e_{jr})=e_{jr}\in S(\lieg).$$ From Lemma \ref{3.1.1} we know that $\{ e_{jr}\mid j\in J, r\in\ii \}$ are lifts of the algebraically independent generators of $S(\lieg)^{\lieg}$. Thereby we easily get $S(\lieg)^{\lieg}\subseteq\on{gr}Z $. Thanks to (\ref{2.1}), we also have $\on{gr}Z\subseteq\on{gr}Z(\lieg)\subseteq S(\lieg)^{\lieg}$, so $\on{gr}Z=\on{gr}Z(\lieg)=S(\lieg)^{\lieg}$. This implies $Z=Z(\lieg)$. 
\end{proof}

The structure of $Z(\lieg)$ is simple when $\lieg$ is a special Lie algebras.

\begin{proposition}\label{3.1.3}
If $\liel$ is a semisimple Lie algebra over $\fk$, then we have $Z(\lieg)=\fk$.
\end{proposition}
\begin{proof}
It is easy to know that $C(\liel)=\left \{ 0 \right \}$ since $\liel$ is a semisimple Lie algebra, we know that $Z(\lieg)=\fk$ by Theorem \ref{3.1.2}.
\end{proof}

\begin{remark}
   If $\lieg=\liel(x)$, the above proof is unaffected, as are all the following conclusions. 
\end{remark}

\subsection{Lie superalgebra}
\begin{lemma}\label{3.2.1}
	If $\liel$ is a Lie superalgebra over $\fk$, then the invariant algebra $S(\lieg)^{\lieg}$ is equal to $\fk[ e_{jr}\mid j\in J, r\in\ii ]$. In particular, if $C(\liel)\neq \left \{ 0 \right \}$, $S(\lieg)^{\lieg}$ is freely generated by $\{ e_{jr}\mid j\in J, r\in \ii \}.$
\end{lemma}
\begin{proof}
    The proof is essentially the same as Lemma \ref{3.1.1}, except that we need to consider the odd elements. We use essentially the same notations as in the argumentation process of Lemma \ref{3.1.1}. 
    
    If $C(\liel)\neq \left \{ 0 \right \}$. It is obvious that $e_{jr}\in S(\lieg)^{\lieg} $ for each $j\in J$ and $r\in \ii.$ Since the elements $\{e_{ir}\mid i\in I,r\in \ii\}$ give a basis for $\lieg$, it follows that $$S(\lieg)=\fk[ e_{jr}\mid j\in I_0, r\in\ii ]\otimes\Lambda[e_{ir}\mid i\in I_1, r\in\ii],$$ where $\Lambda[e_{ir}\mid i\in I_1, r\in\ii]$ is the exterior algebra. Hence, $\fk[ e_{jr}\mid j\in I_0, r\in\ii ]$ and $I(\lieg)$ with both being free polynomial algebras and $S(\lieg)$ is free as an $I(\lieg)$ module with basis $$\left\{ \prod_{i\in I\xg J, r\in \ii} e_{ir}^{	\omega(i,r)} \mid \omega\in {\Omega_1}' \right\},$$
	where $${\Omega_1}'=\left\{\omega: I\xg J\times \ii \ra \fn\left| \begin{matrix}
		\omega(i,r)\ge 0, \forall (i,r)\in I_0\xg J\times\ii \text{ and } \omega(i,r)\in \{0,1\},\forall (i,r)\in I_1\times\ii, \\ \omega(i,r)=0 \text{ for all but finite many }(i,r)\in I\xg J\times \ii
	\end{matrix} \right.\right\}.$$
 The rest of the argument is the same as Lemma \ref{3.1.1} except for (\ref{3.1}).

If $i\in I_0$ or $i\in I_1$, we have
 \begin{multline}
      \on{ad}(e_{is})(f)= \sum_{\omega\in {\Omega_1}'} c_{\omega} \sum_{\substack{j\in I\xg J, r\in \ii\\ \omega(j,r)>0} }\left (\sum_{k\in I\xg J}(-1)^{\on{sgn}(\omega,j,r)} a_{ijk}  \omega(j,r)e_{k,r+s} \right.\\ \left.+\sum_{k\in J}(-1)^{\on{sgn}(\omega,j,r)} a_{ijk}  \omega(j,r)e_{k,r+s} \right )  \times  e_{jr}^{\omega(j,r)-1}\prod_{\substack{j^\prime\in I\xg J, r^\prime\in \ii\\(j^\prime,r^\prime)\neq (j,r)}}e_{j^\prime r^\prime}^{\omega(j^\prime,r^\prime)}. 
	\end{multline}
where $\on{sgn}(\omega,j,r)\in \left\{0,1\right\}$ depends on $\omega, j, r$. Since $(-1)^{\on{sgn}(\omega,j,r)}\ne 0$, we still have $c_{\omega}=0$ by the argument of Lemma \ref{3.1.1}.

The arguments are similar if $C(\liel)=\left \{ 0 \right \}$ which will be omitted here.
\end{proof}

  The proof of the following conclusion will be omitted here, which is identical to the Lie algebra.

\begin{theorem}\label{3.2.2}
If $\liel$ is a Lie superalgebra over $\fk$, then the center $Z(\lieg)$ of $U(\lieg)$ is equal to $\fk[ e_{jr}\mid j\in J, r\in\ii ]$. In particular, if $C(\liel)\neq \left \{ 0 \right \}$, $Z(\lieg)$ is freely generated by $\{ e_{jr}\mid j\in J, r\in\ii \}$.
\end{theorem}

\begin{proposition}\label{3.2.3}
If $\liel$ is a semisimple Lie superalgebra over $\fk$, then we have $Z(\lieg)=\fk$.
\end{proposition}

\section{In prime characteristic}\label{4}

In this section, the characteristic of $\fk$ is prime $p$.

\subsection{Lie algebra}
\begin{lemma}\label{4.1.1}
	If $\liel$ is a Lie algebra over $\fk$, then the invariant algebra $S(\lieg)^{\lieg}  $ is generated by $\{ e_{jr}\mid j\in J, r\in \ii \}$ together with $\lieg^p:=\{ e^p\mid e\in \lieg \}. $ In fact, $S(\lieg)^{\lieg}$ is freely generated by \[\{ e_{jr}\mid j\in J, r\in \ii \}\cup \{ e_{ir}^p\mid i\in I\xg J, r\in \ii \}.\]
		
\end{lemma}
\begin{proof}
 The method of proof is highly depends on \cite[Lemma 3.2]{BT}. It is obvious that $e_{jr}\in S(\lieg)^{\lieg} $ for each $j\in J$ and $r\in \ii.$ Since $p>0,$ we have that $\lieg^p\subseteq S(\lieg)^{\lieg}. $  Let $I(\lieg) $ be the subalgebra of $S(\lieg)^{\lieg}$ generated by $\{e_{jr}\mid j\in J, r\in \ii\}$ and $\lieg^p.$ Since the elements $\{e_{ir}\mid i\in I,r\in \ii\}$ give a basis of $\lieg,$
 it follows that
 \[ S(\lieg)=\fk[ e_{jr}\mid j\in J, r\in\ii ][e_{ir}\mid i\in I\xg J, r\in\ii], \]
	\[ I(\lieg)=\fk[ e_{jr}\mid j\in J, r\in\ii ][e_{ir}^p\mid i\in I\xg J, r\in\ii] \]
	with both being free polynomial algebras. Hence, $S(\lieg)$ is free as an $I(\lieg)$ module with basis $$\left\{ \prod_{i\in I\xg J, r\in \ii} e_{ir}^{	\omega(i,r)} \mid \omega\in \Omega_2 \right\},$$
	where $$\Omega_2=\left\{\omega: I\xg J\times \ii \ra \fn\left| \begin{matrix}
		0\leq\omega(i,r)<p \text{ for all } (i,r)\in I\xg J\times\ii, \\ \omega(i,r)=0 \text{ except finite }(i,r)\in I\xg J\times \ii
	\end{matrix} \right.\right\}.$$
	
	Given $f\in S(\lieg)^{\lieg}, $ we can write
	\[ f=\sum_{\omega\in \Omega_2} c_{\omega} \prod_{i\in I\xg J, r\in \ii}e_{ir}^{\omega(i,r)} \]
	for all $c_{\omega}\in I(\lieg)$, all but finitely many of which are zero. In order to prove $S(\lieg)^{\lieg}\subseteq I(\lieg),$ we have to prove that $c_{\omega}=0$ for all non-zero $\omega.$
	
	Now, fix a non-zero $\omega$ appearing in the expression of $f$ such that $\omega(j,r)>0$ for $j\in I\xg J, r\in \ii.$ If $J=I$, we have that $S(\lieg)=I(\lieg)\subseteq S(\lieg)^{\lieg}\subseteq S(\lieg)$, so $S(\lieg)^{\lieg}=I(\lieg)$. Let $J\subsetneqq I$, since $e_j\not\in C(\liel), $ there is $i\in I\xg J$ such that $[e_i,e_j]\neq 0.$ Namely, $a_{ijk}\neq 0$ for some $k\in I.$
	For each integer $s>\on{deg}_x(f),$ which implies that $s$ is larger than all $r^\prime \in\ii$ such that $\omega(i^\prime,r^\prime)>0$ for $(i^\prime,r^\prime)\in I\xg J\times \ii,$ we have
	\begin{multline}\label{4.2}
      \on{ad}(e_{is})(f)= \sum_{\omega\in \Omega_2} c_{\omega} \sum_{\substack{j\in I\xg J, r\in \ii\\ 0<\omega(j,r)<p} }\left (\sum_{k\in I\xg J} a_{ijk}  \omega(j,r)e_{k,r+s} +\sum_{k\in J} a_{ijk}  \omega(j,r)e_{k,r+s} \right ) \\ \times  e_{jr}^{\omega(j,r)-1}\prod_{\substack{j^\prime\in I\xg J, r^\prime\in \ii\\(j^\prime,r^\prime)\neq (j,r)}}e_{j^\prime r^\prime}^{\omega(j^\prime,r^\prime)}.
	\end{multline}
	
	Thanks to the choice of $s$, the coefficient of
	\[ e_{k,r+s} e_{jr}^{\omega(j,r)-1} \prod_{\substack{j^\prime\in I\xg J, r^\prime\in \ii\\(j^\prime,r^\prime)\neq (j,r)}}e_{j^\prime r^\prime}^{\omega(j^\prime,r^\prime)} \]
	is $ a_{ijk} c_{\omega} \omega(j,r) $ for all $k\in I\xg J,$ and the coefficient of
	\[  e_{jr}^{\omega(j,r)-1} \prod_{\substack{j^\prime\in I\xg J, r^\prime\in \ii\\(j^\prime,r^\prime)\neq (j,r)}}e_{j^\prime r^\prime}^{\omega(j^\prime,r^\prime)} \]
	is $ a_{ijk} c_{\omega} \omega(j,r)e_{k,r+s} $ for all $k\in J.$ They must be zero since $f\in S(\lieg)^{\lieg}.$ As $0<\omega(j,r)<p$ and $ a_{ijk}\neq 0 $ for some $k\in I,$ we conclude that $c_{\omega}=0$ which implies that $S(\lieg)^{\lieg}\subseteq I(\lieg)$.
\end{proof}

\begin{remark}
   If $\lieg$ is the shifted current algebra defined in \cite[3.2]{BT}, using Lemma \ref{4.1.1} we can have that \cite[Lemma 3.2]{BT}.
\end{remark}

Using the restricted structure, define the $p$-center $Z_p(\lieg)$
of $U_p(\lieg)$ to be the subalgebra of $Z(\lieg)$ generated by $x^p-x^{[p]}$ for all $x\in\lieg$. Since the $p$-th power map is $p$-semilinear, we have that
\[Z_p(\lieg)=\fk[e_{ir}^p-e_{ir}^{[p]}=(e_ix^r)^p-e_i^{[p]}x^{rp}\mid i\in I, r\in \ii].
\]

\begin{theorem}\label{4.1.2}
If $\liel$ is a Lie algebra over $\fk$, then the center $Z(\lieg)  $ is generated by $\{ e_{jr}\mid j\in J, r\in \ii \}$ and $Z_p(\lieg). $ In fact, $Z(\lieg)$ is freely generated by 
\begin{equation}\label{4.6}
	\{ e_{jr}\mid j\in J, r\in \ii \}\cup \{ e_{ir}^p-e_{ir}^{[p]}\mid i\in I\xg J, r\in \ii \}.
 \end{equation}
\end{theorem}
\begin{proof}
	The proof, which uses Lemma \ref{4.1.1} and (\ref{2.1}), is similar to the proof of \cite[Theorem 3.4]{BT}. Let $Z$ be the subalgebra of $Z(\lieg)$ generated by (\ref{4.6}). We want to prove that $S(\lieg)^{\lieg}\subseteq\on{gr}Z $. 
 
 For $j\in J, i\in I\xg J$ and $r\in \ii$ we have that $e_{jr}\in \on{F}_{r+1}U(\lieg)$ and $$e_{ir}^p-e_{ir}^{[p]}\in \on{F}_{(r+1)p}U(\lieg),$$ so $\on{gr}(e_{jr})=e_{jr}\in S(\lieg)$ and $\on{gr}(e_{ir}^p-e_{ir}^{[p]})=e_{ir}^p\in S(\lieg)$. From Lemma \ref{4.1.1} we know that (\ref{4.6}) are lifts of the algebraically independent generators of $S(\lieg)^{\lieg}$. Thereby we easily get $S(\lieg)^{\lieg}\subseteq\on{gr}Z $. Thanks to (\ref{2.1}), we also have $\on{gr}Z\subseteq\on{gr}Z(\lieg)\subseteq S(\lieg)^{\lieg}$, so $\on{gr}Z=\on{gr}Z(\lieg)=S(\lieg)^{\lieg}$. This implies $Z=Z(\lieg)$. 
\end{proof}

\begin{remark}\label{4.1.04}
 If $\lieg$ is the shifted current algebra defined in \cite[3.2]{BT}, using Theorem \ref{4.1.2} we can have that \cite[Theorem 3.4]{BT}.   
\end{remark}

Similar to the case of characteristic zero, we have the following proposition.

\begin{proposition}\label{4.1.3}
If $\liel$ is a semisimple Lie algebra over $\fk$, then we have $Z(\lieg)$ is freely generated by $$\{ e_{ir}^p-e_{ir}^{[p]}\mid i\in I, r\in \ii \}.$$
\end{proposition}
\begin{proof}
It is easy to know that $C(\liel)=\left \{ 0 \right \}$ since $\liel$ is a semisimple Lie algebra, we know that $$Z(\lieg)=Z_p(\lieg)=\fk[e_{ir}^p-e_{ir}^{[p]}\mid i\in I, r\in \ii]$$ by Theorem \ref{4.1.2}.
\end{proof}
\subsection{Lie superalgebra}
\begin{lemma}\label{4.2.1}
	If $\liel$ is a Lie superalgebra over $\fk$, then the invariant algebra $S(\lieg)^{\lieg}  $ is generated by $\{ e_{jr}\mid j\in J, r\in \ii \}$ together with $\lieg_{\overline{0}}^p:=\{ e^p\mid e\in \lieg_{\overline{0}} \}. $ In fact, $S(\lieg)^{\lieg}$ is freely generated by 
 \[\{ e_{jr}\mid j\in J, r\in \ii \}\cup \{ e_{ir}^p\mid i\in I_0\xg J, r\in \ii \}.
	\]
\end{lemma}
\begin{proof}
    The proof is essentially the same as Lemma \ref{4.1.1}, except that we need to consider the odd elements. We use essentially the same notations as in the argumentation process of Lemma \ref{4.1.1}.
    
    It is obvious that $e_{jr}\in S(\lieg)^{\lieg} $ for each $j\in J$ and $r\in \ii.$ Since $p>0,$ we have that $\lieg_{\overline{0}}^p\subseteq S(\lieg)^{\lieg}. $  Let $I(\lieg) $ be the subalgebra of $S(\lieg)^{\lieg}$ generated by $\{e_{jr}\mid j\in J, r\in \ii\}$ and $\lieg_{\overline{0}}^p.$ Since the elements $\{e_{ir}\mid i\in I,r\in \ii\}$ give a basis of $\lieg,$
 it follows that $$S(\lieg)=\fk[ e_{jr}\mid j\in I_0, r\in\ii ]\otimes\Lambda[e_{ir}\mid i\in I_1, r\in\ii],$$ where $\Lambda[e_{ir}\mid i\in I_1, r\in\ii]$ is the exterior algebra. Hence, $\fk[ e_{jr}\mid j\in I_0, r\in\ii ]$ and $I(\lieg)$ with both being free polynomial algebras and $S(\lieg)$ is free as an $I(\lieg)$ module with basis $$\left\{ \prod_{i\in I\xg J, r\in \ii} e_{ir}^{	\omega(i,r)} \mid \omega\in {\Omega_2}' \right\},$$
	where $${\Omega_2}'=\left\{\omega: I\xg J\times \ii \ra \fn\left| \begin{matrix}
		0\leq \omega(i,r)<p, \forall (i,r)\in I_0\xg J\times\ii \text{ and } \omega(i,r)\in \{0,1\},\forall (i,r)\in I_1\times\ii, \\ \omega(i,r)=0 \text{ for all but finite many }(i,r)\in I\xg J\times \ii
	\end{matrix} \right.\right\}.$$
 The rest of the argument is the same as Lemma \ref{4.1.1} except for (\ref{4.2}).

If $i\in I_0$ or $i\in I_1$, we have
 \begin{multline}
      \on{ad}(e_{is})(f)= \sum_{\omega\in {\Omega_2}'} c_{\omega} \sum_{\substack{j\in I\xg J, r\in \ii\\ \omega(j,r)>0} }\left (\sum_{k\in I\xg J}(-1)^{\on{sgn}(\omega,j,r)} a_{ijk}  \omega(j,r)e_{k,r+s} \right.\\ \left.+\sum_{k\in J}(-1)^{\on{sgn}(\omega,j,r)} a_{ijk}  \omega(j,r)e_{k,r+s} \right )  \times  e_{jr}^{\omega(j,r)-1}\prod_{\substack{j^\prime\in I\xg J, r^\prime\in \ii\\(j^\prime,r^\prime)\neq (j,r)}}e_{j^\prime r^\prime}^{\omega(j^\prime,r^\prime)}.  
	\end{multline}
where $\on{sgn}(\omega,j,r)\in \left\{0,1\right\}$ depends on $\omega, j, r$. Since $(-1)^{\on{sgn}(\omega,j,r)}\ne 0$, we still have $c_{\omega}=0$ by the argument of Lemma \ref{4.1.1}.
\end{proof}

\begin{remark}
  In fact, $\{ e_{ir}^p\mid i\in I\xg J, r\in \ii \}=\{ e_{ir}^p\mid i\in I_0\xg J, r\in \ii \}$ and $\lieg^p=\lieg_{\overline{0}}^p$. If $\lieg=\mathfrak{gl}_{m\mid n}[x]$, using Lemma \ref{4.2.1} we can have that \cite[Lemma 2.1]{CH}. 
\end{remark}

Similar to Lie algebra, the $p$-center $Z_p(\lieg)$ of $U(\lieg)$ is the subalgebra of $Z(\lieg)$ generated by $x^p-x^{[p]}$ for all $x\in \lieg_{\overline{0}}$, we still have that
$$Z_p(\lieg)=\fk[e_{ir}^p-e_{ir}^{[p]}=(e_ix^r)^p-e_i^{[p]}x^{rp}\mid i\in I_0, r\in \ii].$$

The proof of the following conclusion will be omitted here, which is identical to the Lie algebra.

\begin{theorem}\label{4.2.2}
If $\liel$ is a Lie superalgebra over $\fk$, then the center $Z(\lieg)  $ is generated by $\{ e_{jr}\mid j\in J, r\in \ii \}$ and $Z_p(\lieg). $ In fact, $Z(\lieg)$ is freely generated by 
	$$\{ e_{jr}\mid j\in J, r\in \ii \}\cup \{ e_{ir}^p-e_{ir}^{[p]}\mid i\in I_0\xg J, r\in \ii \}.$$
\end{theorem}

\begin{proposition}\label{4.2.3}
If $\liel$ is a semisimple Lie superalgebra over $\fk$, then we have $Z(\lieg)$ is freely generated by $$\{ e_{ir}^p-e_{ir}^{[p]}\mid i\in I_0, r\in \ii \}.$$
\end{proposition}

\begin{remark}\label{4.2.10}
 If $\lieg=\mathfrak{gl}_{m\mid n}[x]$, using Theorem \ref{4.2.2} we can have that \cite[Theorem 2.3]{CH}.    
\end{remark}

\bigskip
\noindent
\textbf{Acknowledgments.}
The second named author would like to give some special thanks to Ke Ou for his valuable insight and guidance with this paper.

\end{document}